\documentclass[11pt,a4paper]{article}

\usepackage{mathtools}
\usepackage{amssymb,amsfonts,amsthm}
\usepackage[margin=3.4cm]{geometry}
\usepackage{hyperref}
\hypersetup{colorlinks=true,allcolors=black}
\usepackage[dvips]{graphicx}
\usepackage{standalone}
\usepackage{xcolor,tikz,tikz-cd}
\usepackage{todonotes}
\usepackage{cancel}
\usepackage{framed}
\definecolor{refkey}{rgb}{0.4,1,0.4}
\definecolor{labelkey}{rgb}{0.5,0.5,0.8}

\theoremstyle{plain}
\begingroup
\newtheorem{teo}{Theorem}[section]
\newtheorem{prop}[teo]{Proposition}

\newtheorem{lem}[teo]{Lemma}
\endgroup

\theoremstyle{definition}
\begingroup
\newtheorem{defin}[teo]{Definition}
\newtheorem{rem}[teo]{Remark}
\newtheorem{exam}[teo]{Example}

\endgroup

\numberwithin{equation}{section}

\newcommand{\R}{\ensuremath{\mathbb{R}}}

\newcommand{\id}{\ensuremath{\mathrm{id}}}

\newcommand{\cL}{\ensuremath{\mathcal{L}}}
\newcommand{\cD}{\ensuremath{\mathcal{D}}}

\newcommand{\cG}{\ensuremath{G}}
\newcommand{\X}{\ensuremath{\mathfrak{X}}}

\newcommand{\Met}{\ensuremath{\mathrm{Met}}}

\newcommand{\dist}{\ensuremath{\,\mathrm{dist}}}

\DeclareMathOperator*{\Div}{div}

\newcommand{\edit}[1]{#1}

\newcommand{\vp}{\varphi}                          

\usepackage[normalem]{ulem}                     

\definecolor{dgreen}{rgb}{0.14,0.4,0.14}

\title{\textbf{Shape analysis via gradient flows \\ on diffeomorphism groups}}
\author{Tracey Balehowsky$^1$ \and Carl-Joar Karlsson$^2$ \and Klas Modin$^{2,}$\footnote{Corresponding author: \href{mailto:klas.modin@chalmers.se}{\texttt{klas.modin@chalmers.se}}}}

\begin{document}
\maketitle

\noindent
$^1$ Department of Mathematics \& Statistics, University of Calgary, Canada\\[0.2em] 
$^2$ Department of Mathematical Sciences, Chalmers University of Technology and University of Gothenburg, Sweden


\begin{abstract}

We study a Riemannian gradient flow on Sobolev diffeomorphisms for the problem of image registration. The energy functional quantifies the effect of transforming a template to a target, while also penalizing non-isometric deformations. The main result is well-posedness of the flow. We also give a geometric description of the gradient in terms of the momentum map.




\medskip


\medskip

\noindent\textbf{AMS 2020:} 58D05 (primary); 35F25 68U10 (secondary)

\end{abstract}

\section{Introduction}


Shape analysis comprises mathematical models devised to understand the nature of shapes and transformations of shapes. 
Over the last thirty years, it has undergone rapid development, with applications in biology, physics, computer graphics, design, computer vision and medical imaging.
In particular, based on ideas going back to the evolutionary biologist D'Arcy Wentworth Thompson~\cite{Th1917}, Grenander and others  initiated \emph{computational anatomy} (cf.~\cite{Gr1993,grenander2007,younes2010}).
Here one uses shape analysis to characterize, find, or understand disease via abnormal anatomical deformations of organs such as the brain or the lungs. 

In shape analysis, new shapes are obtained by deforming a template.
Often the deformation is achieved by the action of a group of diffeomorphisms.
As a central example, consider a template function $I_0$ on a Riemannian manifold $(M,g)$. 
This template may describe gray-scale values for an image generated by magnetic resonance imaging (MRI) or computed tomography (CT). 
If $\varphi$ is a diffeomorphism, the deformation of the template is $I_0\circ\varphi^{-1}$. This new function may model an image taken at a different time or from a different viewing angle or both. 

Let $\cD(M)$ denote the group of all diffeomorphisms of $M$.
The \emph{template matching problem} is to find a diffemorphism $\varphi\in\cD(M)$ that deforms a template $I_0\circ\varphi^{-1}$ to match a given target $I_1$. 
In the generic situation, this problem is ill-posed for two reasons.
First, the action of $\cD(M)$ on the space of functions is not transitive: from a given template only a subset of targets can be reached, namely those belonging to the group orbit \edit{of the template function}.
Second, even for the relaxed problem, where one tries to find the best fit with respect to a similarity measure (typically a norm), the problem is ill-posed since there are no restrictions on how ``wild'' the diffeomorphism can be.
The resolution is to solve a regularized minimization problem for a functional $E\colon\cD(M)\to \mathbb{R}$ of the form
%
\begin{equation}
\label{eq:sim-and-dist}
    E(\varphi)= \frac{1}{2}\lVert I_0\circ\varphi^{-1}-I_1\rVert_{L^2}^2 + \sigma R(\varphi), \qquad \sigma \geq 0.
\end{equation}
Here, $\lVert\cdot \rVert_{L^2}$ is the $L^2$ norm on the space of functions assuring that gross features are matched, and $R\colon \cD(M)\to \mathbb{R}^+$ is a regularization functional assuring that $\varphi$ remains well-behaved.
The larger the parameter $\sigma$, the more regularization is imposed on the problem.
The standard regularization choice is to equip $\cD(M)$ with a right-invariant Riemannian metric and then take $R(\varphi)=\dist(\varphi,\id)$, where $\id$ is the identity transformation and $\dist$ is the induced Riemannian distance on $\cD(M)$. 
This setting is called \emph{large deformation diffeomorphic metric matching} (LDDMM)  \cite{Tr1995,DuGrMi1998, Joshi2000landmark,Be2003,beg2005computing,MaMcMoPe2013}.
It closely resembles Arnold's~\cite{arnold1966geometrie} geometric description of an incompressible ideal fluid as a geodesic equation on the group of volume-preserving diffeomorphisms.
Results on existence of minimizers are given by Trouvé~\cite{Tr1995}.
It is also possible to formulate multiscale versions of LDDMM, which enables convergence results for $\sigma\to 0$ \cite{MoNaRo2019}.

To minimize $E(\varphi)$ one has to take into account that $\varphi$ should be a diffeomorphism.
The best way is to generate $\varphi$ by integration of a time-dependent vector field $v_t = v(t,\cdot)$ on $M$.
That is, we take $\varphi$ to be the end-point $\gamma(1)$ for the path $\gamma(t)$ in $\cD(M)$ defined by
\begin{equation}\label{eq:reconstruction_vf}
    \dot\gamma = v_t\circ\gamma, \qquad t\in [0,1]. 
\end{equation}
Due to the right-invariance of the Riemannian metric on $\cD(M)$, the LDDMM problem is naturally formulated as a minimization problem over $v$
\begin{equation}\label{eq:lddmm_min}
    \min_{v} \left( \lVert I_0\circ\gamma(1)^{-1}-I_1\rVert_{L^2}^2  + \sigma\int_{0}^1\langle v_t,v_t\rangle dt\right)
\end{equation}
subject to equation \eqref{eq:reconstruction_vf}.
Here, $\langle\cdot,\cdot\rangle$ is the inner product on the space of vector fields \edit{$\mathfrak{X}(M)=T_\id \cD(M)$} defining the right-invariant Riemannian structure on $\cD(M)$.
Since $\gamma(t)$ is generated by $v$, the problem can be written entirely in terms of the time-dependent function $I(t,\cdot) = I_0\circ\gamma(t)^{-1}$ and $v$ as
\begin{equation}\label{eq:lddmm_min_reduced}
    \min_{v,I} \left( \lVert I(1)-I_1\rVert_{L^2}^2  + \sigma\int_{0}^1\langle v_t,v_t\rangle dt\right)
\end{equation}
subject to the constraint
\[
    \dot I + v_t\cdot \nabla I = 0 , \qquad I(0) = I_0 .
\]
For details on this setting, including the natural generalization to arbitrary Lie groups, see \cite{Bruveris2011,BrHo2013} and references therein.
Geometrically, the minimization problem~\eqref{eq:lddmm_min_reduced} is closely related to the dynamical formulation of optimal transport~\cite{BeBr2000}.
Although the problem~\eqref{eq:lddmm_min_reduced} is geometrically intuitive, it is numerically better to work with \eqref{eq:lddmm_min}.
That is, it is better to keep track of the mesh deformation, as a discretization of $\gamma(t)$, rather than iteratively applying many small deformations, as a discretization of $I(t)$.


A mathematically beautiful feature of LDDMM is its connection to hydro\-dynamic-type partial differential equations (PDEs).
Indeed, take the inner product on vector fields to be Sobolev $H^k$
\begin{equation}\label{eq:inertia_operator}
    \langle v_t,v_t\rangle = \int_M v_t\cdot \underbrace{(1-\alpha\Delta)^k}_A v_t \, d\mu, \qquad \alpha>0, \; k\in\mathbb{N},
\end{equation}
where $A = (1-\alpha\Delta)^k$  is the \emph{inertia operator}\footnote{The results presented in this paper hold when $A$ is a general elliptic differential operator of order $k\ge2$. We choose $A = (1-\alpha\Delta)^k$ for simplicity of computation.} (the name is borrowed from the language of geometric mechanics, cf.~\cite{arnold1989mathematical}). 
A requirement for $v$ to be optimal is that it fulfills the \emph{EPDiff equation}~\cite{Mu1998,MiTrYo2002,HoMa2005}
%
\begin{equation}\label{eq:epdiff}
    \dot m+Dm\cdot v+Dv^\top m+\Div(v)m=0,\quad m=Av .
\end{equation}
Geometrically these equations arise from the geodesic equation on $\cD(M)$ with respect to the right-invariant Riemannian metric on $\cD(M)$ induced by the inner product~\eqref{eq:inertia_operator}.
In one spatial dimension, the equations \eqref{eq:epdiff} become the Camassa--Holm model for shallow water motion~\cite{CaHo1993}.
%

For all its beauty, the LDDMM framework does have a drawback: it is computationally expensive because the Riemannian distance on $\cD(M)$ lacks a closed-form expression.
Therefore, to numerically solve the minimization problem \eqref{eq:lddmm_min} one has to either iterate over a discretization of entire paths of vector fields or apply initial value shooting over a discretization of the EPDiff equation.
Both approaches have the same computational complexity per iteration: $\mathcal O(nm)$, where $n$ and $m$ are the numbers of spatial and temporal discretization points.

There is a way to avoid the high computational complexity of LDDMM.
Indeed, if the variational derivative of the energy function \eqref{eq:sim-and-dist} on $\cD(M)$ is given by explicit formulae one can consider the Riemannian gradient flow
%
\begin{equation}
\label{eq:grad-flow-intro}
    \dot{\varphi} =  -\nabla E(\varphi)
\end{equation}
for some initial condition $\varphi(0)=\varphi_0\in \cD(M)$.
Here, $\nabla E$ denotes the gradient on $\cD(M)$ with respect to the right-invariant metric defined by the inner product~\eqref{eq:inertia_operator}.
Based on this, a numerical method is obtained via gradient descent (cf.~\cite[\textsection 11.3]{younes2010}).
The simplest case, where $E$ does not contain a regularization term, i.e., $\sigma=0$ in~\eqref{eq:sim-and-dist}, is called \textit{greedy matching} \cite{christensen1996deformable,Trouve1998,beg2005computing}. 
Typically greedy matching algorithms run indefinitely, yielding increasingly complicated diffeomorphisms that in vain try to achieve $I_0\circ\varphi^{-1}=I_1$.
Greedy matching is therefore numerically ill-posed.

In this paper, we study the registration problem using a gradient method regularized on the space $\Met(M)$ of Riemannian metrics on $M$.
Precisely, we penalize the distance between the original metric $g$ and its push-forward $\varphi_*g$ using the $L^2$-distance in the space of symmetric 2-forms,
\begin{equation}\label{eq:regularizer}
    R(\varphi)= \frac{1}{2}\lVert\varphi_*g-g\rVert_{L^2}^2.
\end{equation}
By computing the variational derivative of $R(\varphi)$ we can avoid the costly algorithms associated with regularization via the Riemannian distance as in~\eqref{eq:lddmm_min}.
Still, we can retain the geometric properties, in particular that diffeomorphisms are generated by vector fields.
The explicit form of the gradient flow~\eqref{eq:grad-flow-intro} for the functional \eqref{eq:sim-and-dist} with regularization~\eqref{eq:regularizer} is \edit{a flow $t\mapsto \varphi(t) \in \cD(M)$ given in terms of $\varphi$ and the generating vector field $v$ as} the non-linear partial differential equation
\begin{equation}\label{eq:gradient-flow-explicit}
    \begin{aligned}
        &\dot\varphi = v\circ\varphi, \\
        &(1-\alpha\Delta)^k v = (I-I_1)\nabla I - \sigma \operatorname{div}_{h}(h-g) ,
    \end{aligned}
\end{equation}
where $k\ge 1$, $I \coloneqq I_0\circ\varphi^{-1}$, $h \coloneqq \varphi_*g$, \edit{and $(h,p)\to \operatorname{div}_h(p)$ is a bi-linear divergence-type differential operator (see Proposition~\ref{prop:gradient_flow} for details)}.

From a geometric point of view, the regularization choice~\eqref{eq:regularizer} is closely related to the LDDMM setting.
Indeed, the $L^2$-metric on the space of Riemannian metrics induces, via pull-back, an $H^1$-metric on the space of diffeomorphisms (cf.~\cite{modin2015generalized,BaJoMo2015}); see Figure~\ref{fig:3mfld} for an illustration of the geometry.
Since the action of $\cD(M)$ on the space of Riemannian metrics is non-transitive, this means that the regularization \eqref{eq:regularizer} is an outer distance corresponding to an $H^1$ Riemannian distance on $\cD(M)$.
Recently, the idea to act on Riemannian metrics is also explored for applications in computational anatomy of the brain~\cite{CaDaZhBaFlJo2021}.

\begin{figure}
	\centering
    \input{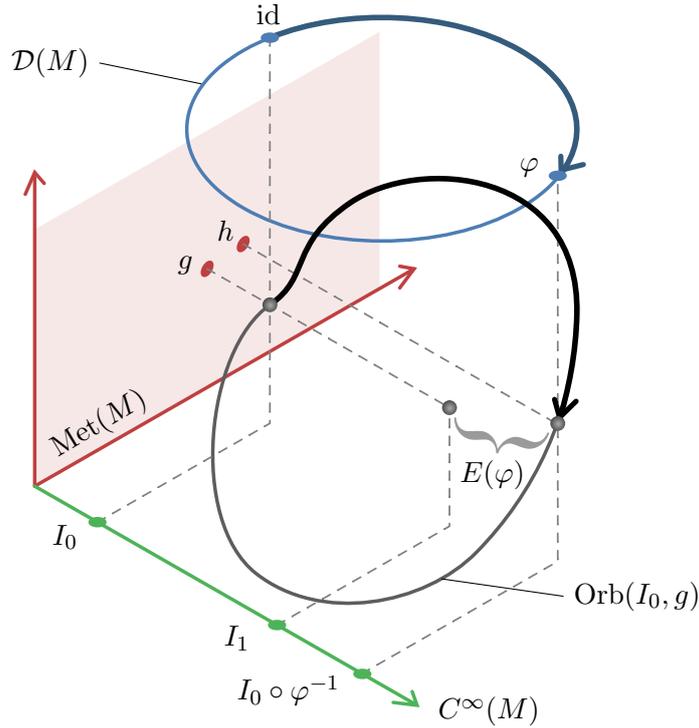}
	\caption{\label{fig:3mfld} \edit{Illustration of the geometric setup. The group $\cD(M)$ acts on $(I_0,g) \in C^\infty(M)\times\mathrm{Met}(M)$ which induces the orbit $\mathrm{Orb}(I_0,g)$. The functional $E(\varphi)$ measures the distance from the point $(I_0\circ\varphi^{-1},\varphi_*g)$ on the orbit to the target $(I_1,g)$ typically not on the orbit. The flow evolves on the group $\cD(M)$ or, equivalently, on the orbit $\mathrm{Orb}(I_0,g)$, such that it strives to minimize the functional $E(\varphi)$.}}
\end{figure}

Our main result is the following existence theorem for the gradient flow~\eqref{eq:gradient-flow-explicit}.
\begin{teo}\label{teo:local-existence}
Let $(M,g)$ be a $C^\infty$-smooth, oriented, compact Riemannian manifold without boundary. 
Furthermore, let $I_0,I_1 \in C^\infty(M)$.
For $s>2+\frac{1}{2}\dim M$, let $\cD^s(M)$ denote the Hilbert manifold obtained by completion of $\cD(M)$ in the Sobolev $H^s$ topology of maps (cf.~\cite{Pa1968}).
Then the Cauchy problem~\eqref{eq:gradient-flow-explicit} is locally well posed on $\cD^s(M)$:
\begin{enumerate}
    \item For each initial datum $\varphi_0\in\cD^s(M)$ there exists a maximal $T>0$ and unique curves $\varphi\colon[0,T)\to\cD^s(M)$ and $v\colon[0,T)\to T_\id\cD^s(M)$ with $\varphi(0) = \varphi_0$ that fulfill equations~\eqref{eq:gradient-flow-explicit}.

    \item For $t\in [0,T)$ the solution $\varphi(t)$ depends smoothly on $\varphi_0$ (in the Hilbert manifold topology of $\cD^s(M)$).
\end{enumerate}
Furthermore, if $k\geq s$ then $T=\infty$ so the flow is globally well posed.
\end{teo}

The key to the proof of Theorem~\ref{teo:local-existence} is to show that the system~\eqref{eq:gradient-flow-explicit} is an ordinary differential equation defined by a smooth vector field on the infinite-dimensional Hilbert manifold $\cD^s(M)$.
Once this is established, local existence follows from the Picard-Lindelöf theorem on Banach manifolds (cf.~\cite{lang1999}).
Our approach can be viewed as a gradient flow analog of the approach by Ebin and Marsden~\cite{ebin1970groups} for geodesic flows on groups of \edit{diffeomorphisms}, in turn based on Arnold's~\cite{arnold1966geometrie} geometric description of hydrodynamics (see~\cite{KhMiMo2019,BaMo2020,KhMiMo2020} for modern developments).

To prove smoothness of the vector field on $\cD^s(M)$ we need to understand the geometry underlying the equations~\eqref{eq:gradient-flow-explicit}.
The following result gives the geometric structure in terms of an infinite-dimensional momentum map (cf.~\autoref{sec:gradflows}).



\begin{prop}\label{teo:gradient-functional}
With $\cD^s(M)$ as in Theorem~\ref{teo:local-existence} \edit{and with $\Met^{s-1}(M)$ the Sobolev $H^{s-1}$ completion of $\Met(M)$}, let $Q=H^{s}(M,\mathbb{R})\times\Met^{s-1}(M)$.
Then the functional $E(\varphi)$ for the gradient flow~\eqref{eq:gradient-flow-explicit} is of the form
\[
    E(\varphi) = f(I_0\circ\varphi^{-1},\varphi_*g)
\]
with $f\colon Q\to\mathbb{R}$ given by
\[
    f(I, h) = \frac{1}{2}\lVert I-I_1 \rVert_{L^2}^2 + \frac{\sigma}{2}\lVert h-g \rVert_{L^2}^2.
\]
Furthermore, let $\mathfrak{X}^{s-2}(M)$ denote vector fields of Sobolev type $H^{s-2}$ \edit{and $S_{2,0}^{s-1}(M)$ the space of symmetric $(2,0)$-tensor fields of Sobolev type $H^{s-1}$}.
Then the gradient of $E$ with respect to the right-invariant Riemannian metric on $\cD^s(M)$ defined by~\eqref{eq:inertia_operator} is
\[
    \nabla E(\varphi) = v\circ\varphi, \qquad v = -(1-\alpha\Delta)^{-k}J\left( \varphi_*I_0,\varphi_*g, \frac{\delta f}{\delta I}, \frac{\delta f}{\delta h} \right)
\]
where \edit{$J\colon T^*Q \to \mathfrak{X}(M)^*$, identified as a mapping $Q\times H^s(M,\mathbb{R})\times S_{2,0}^{s-1}(M) \to \mathfrak{X}^{s-2}(M)$,} is given by
\[
    J(I, h, P, p) = -P\nabla I + 2 \operatorname{div}_h(p)
\]
and is the momentum map for the action of $\cD^s(M)$ on the co-tangent bundle $T^*Q$.
\end{prop}




Our paper is organized as follows. 
In \autoref{sec:gradflows} we consider gradient flows on Lie groups in an abstract setting, for a generic right-invariant Riemannian metric and functionals obtained via lifting of a group action. 
The group action gives rise to the momentum map in an abstract version of Proposition~\ref{teo:gradient-functional}.
Thereafter we give global existence results in the abstract setting.
In \autoref{section:3} we change perspective to the specific infinite-dimensional case above, where the group consists of Sobolev diffeomorphisms of a compact Riemannian manifold.
All results from the abstract setting are valid in the Hilbert category of Riemannian manifolds and Lie groups.
However, the Sobolev diffeomorphisms setting does not quite fit in the abstract setting, for two reasons.
First, although the space of Sobolev diffeomorphisms is a Hilbert manifold, it is not a Lie group with respect to its manifold structure (it is only a topological group).
Second, the right-invariant Riemannian metric we use is not necessarily strong with respect to the manifold structure.
Despite these obstructions, enough structure survives to recover, under adequate conditions, the results in the abstract setting.
In \autoref{section:gradflow-diffeos} we prove this in a series of lemmata, building up to the proof of Theorem~\ref{teo:local-existence} in \autoref{exist-and-smoothness-section}.




\medskip

\textbf{Acknowledgments.} 
This work was supported by the Swedish Research Council (K.M.\ grant number 2017-05040; C-J.K.\ grant number 2018-03873) and the Knut and Alice Wallenberg Foundation (K.M.\ grant number WAF2019.0201)


\section{Gradient flows on Lie groups \label{sec:gradflows}}

Here we study gradient flows on Lie groups for functionals lifted via a group action.
The gradient in this setting carries a close relationship to the momentum map in classical mechanics. 
Momentum maps generalize the notions of linear and angular momenta (cf.~\cite{arnold1989mathematical,marsden1999mechanics}).

To begin, let $\cG$ be a Lie group that acts on a manifold $Q$ from the left,
\begin{align*}
    \ell \colon \cG\times Q \to Q, \qquad (g,q) \mapsto \ell(g,q)
\end{align*}
We shall use the notation $g.q \coloneqq \ell(g,q)$. 
Recall the orbit $\cG.q=\{\ell(g,q)\in Q\mid g\in\cG\}$ of a point $q\in Q$. 
In shape analysis, the orbit $G.q$ represents all possible deformations of the template $q$.

Keeping $q\in Q$ fixed, the differential of the action map $\ell(g,q)$ at the identity $e$ defines the \textit{infinitesimal} action $\mathfrak g\times Q\to TQ,\ (\xi,q)\mapsto \xi.q$, where $\mathfrak g$ is the Lie algebra of $\cG$. 
We may fix $\xi\in\mathfrak{g}$ to obtain a vector field $\xi_\ell\colon Q\to TQ$ that generates the group action.




\begin{defin}[Momentum maps of cotangent lifted actions]
Given a left action of $\cG$ on $Q$, the \emph{cotangent lifted momentum map} $J\colon T^*Q\to\mathfrak{g}^*$ is defined by
\begin{equation} 
\label{def:cotangent-momentum-map-intro}
\langle J(q,p),\xi\rangle = \langle p,\xi_\ell(q)\rangle, \qquad \forall\,\xi\in\mathfrak{g}.
\end{equation}
\end{defin}

\begin{exam}
    Let $\mathrm{SO}(3)$ act on $\R^3$ by matrix multiplication. The Lie algebra $\mathfrak{so}(3)$ is identified with $\R^3$ via the ``hat isomorphism''
    \[
    \begin{pmatrix}x\\y\\z
    \end{pmatrix}
    =\omega\mapsto \hat{\omega}=
    \begin{pmatrix}
    0 & -z &y\\
    z & 0 & -x\\
    -y & x & 0
    \end{pmatrix}.
    \]
    The infinitesimal generator of the action is given by $\hat{\omega} q = \omega\times q$. 
    It follows that $J\colon T^*\R^3\to \mathfrak{so}(3)^*\simeq \R^3$ is defined by $\langle J(q,p),\omega\rangle=p\cdot \hat{\omega}q =\omega\cdot(q\times p)$ for all $\omega\in\mathbb{R}^3$.
    Thus, $J(q,p)=q\times p$, the angular momentum vector in classical mechanics. 
\end{exam}

The group $\cG$ can be equipped with a Riemannian metric. 
If this metric is right invariant, then each orbit of the left action (assuming it is a manifold) also inherits a Riemannian structure. 
To see this, let $(\,,)$ be a Riemannian metric on $\cG$ and assume that it is right invariant, that is,
\begin{equation}
\label{eq:leftinv-metric}
(\xi,\eta)_e = (\xi.g,\eta.g)_g\qquad \text{for all } g\in \cG,\ \eta,\xi\in\mathfrak{g},
\end{equation}
where $TR_g:\xi\mapsto \xi.g$ is the tangent map of the right action $R_g:\cG\to\cG, h\mapsto hg.$ 
Such a metric is defined on the entire group once it is defined at the identity $e\in\cG.$
We write the inner product at the identity via an operator $A:\mathfrak{g}\to\mathfrak{g}^*$
\begin{equation}
\label{eq:A-product-definition}
	(\xi,\eta)_e = \langle A\xi,\eta\rangle.
\end{equation}
Throughout, $A$ is called the \emph{inertia operator}, following the nomenclature of classical mechanics \cite{arnold1989mathematical}.
We use the notation $\lVert \xi \rVert_A \coloneqq \sqrt{\langle A\xi,\eta\rangle}$ for the corresponding norm.

Any Riemannian metric provides an isomorphism between vectors and covectors via the \emph{flat isomorphism}. 
The gradient of a smooth function $E$ is defined as the vector whose corresponding covector is the differential $dE$. 
Given the Riemannian structure~(\ref{eq:leftinv-metric}), the gradient $\nabla E$ on $\cG$ at $g\in\cG$ is given by
\begin{equation}
\label{eq:gradient-definition}
    (\nabla E(g),\eta.g)_g=\langle dE_g,\eta.g\rangle 
\end{equation}
for all $\eta\in\mathfrak{g}.$ The following theorem provides the geometric meaning of the gradient on the orbit of a left action $\cG\times Q\to Q$.

\begin{prop}\label{teo:gradient-intro}
Assume that the Lie group $\cG$ is endowed with a right invariant metric.
Let $J:T^*Q\to\mathfrak{g}^*$ be the momentum map associated to the cotangent lifted  left action $\cG\times Q\ni(g,q)\mapsto g.q\in Q$.
Fix $q\in Q$ and let $E(g)=f(g.q)$ for some smooth function $f$ on $Q$. 
Then the gradient of $E$ is given by
\begin{equation}\label{eq:gradient-expression}
	\nabla E(g) = TR_{g}\,\eta=\eta.g,\qquad\text{where}\quad A\eta=J\big(g.q,df(g.q)\big).
\end{equation}
\end{prop}

\begin{proof}
Let $\xi\in\mathfrak{g}$. 
By the chain rule,
\begin{equation}
\langle dE(g),\dot{g}\rangle=\langle df(g.q),\xi_\ell(g.q)\rangle.
\end{equation} 
By Definition~\ref{def:cotangent-momentum-map-intro},
\begin{equation}
\langle df(g.q),\xi_\ell(g.q)\rangle =\langle J\big(g.q, df(g.q)\big),\xi\rangle,
\end{equation}
The construction of the metric via $(\xi,\eta)_e \coloneqq \langle A\xi,\eta\rangle$ and right invariance,~(\ref{eq:leftinv-metric}), gives
\begin{equation}
\langle J(g.q,df(g.q)),\xi\rangle = (A^{-1}J(g.q,df(g.q)),\xi)_e= (TR_g A^{-1}J(g.q,df(g.q)),\xi.g)_g
\end{equation}
for $\xi.g=TR_g\xi.$ This completes the proof.
\end{proof}

\begin{exam}
	We continue the previous example.
	Consider the function on $\mathbb{R}^3$ given by $f(x)=\|x-x_1\|^2$, where $x_1$ is a fixed target and the norm is given by the standard Euclidean inner product. 
	Then $E:\mathrm{SO}(3)\to\R$ may be defined via $E(\hat{\omega})=f(\hat{\omega}x_0)$ some fixed template $x_0$.
	Let $A$ be an inertia operator.
	Then $df(x)= 2(x-x_1)$, so
	\begin{equation}
		J\left(\hat{\omega}x_0,df(\hat{\omega}x_0)\right)= 2(\hat{\omega}x_0)\times (\hat{\omega}x_0-x_1).
	\end{equation}
	Therefore, $\nabla E(\hat{\omega})=A^{-1}2(\hat{\omega}x_0)\times (\hat{\omega}x_0-x_1),$ where the inverse of $A$ acts on the vector $2(\hat{\omega}x_0)\times (\hat{\omega}x_0-x_1)$.
\end{exam}

\begin{rem}
    From the point-of-view of numerical methods, an interesting open problem is to develop integration schemes on matrix  Lie groups that preserve the gradient flow structure discussed here.
    To this end, the isospectral integrators of Modin and Viviani~\cite{MoVi2020c} can likely be adopted.
\end{rem}

\subsection{Well-posedness of the flow}
For the setting in Proposition~\ref{teo:gradient-intro}, consider the gradient flow
\begin{equation}
    \dot g = -\xi. g, \qquad A\xi = J(g.q_0, df(g.q_0))
\end{equation}
where $q_0\in Q$ is some fixed template.
Local in time existence follows by standard ODE theory (cf.~\cite{lang1999}), since the vector field defining the dynamics is locally Lipschitz continuous.
In this section, we show that the flow is also globally defined.
To see this we work with the distance induced by the Riemannian metric on $\cG$.
Recall its definition
\[
    d(g_0,g_1) = \inf_{\gamma} \int_0^1 \lVert \dot\gamma(t)\cdot\gamma(t)^{-1} \rVert_{A} dt
\]
where the infimum is taken over all all smooth curves $\gamma\colon[0,1]\to\cG$ such that $\gamma(0) = g_0$ and $\gamma(1) = g_1$.

\begin{lem}\label{lem:estimate}
    Assume that $\gamma\colon [0,\epsilon) \to \cG$ is a $C^1$ solution curve to the gradient flow
    \begin{equation}\label{eq:gradient-flow-finite-dim}
        \dot\gamma = -\nabla E(\gamma), \quad \gamma(0) = g_0\in\cG
    \end{equation}
    for some function $E\colon\cG\to \mathbb{R}$.
    Let $v(t) = \dot\gamma(t)\cdot \gamma(t)^{-1}$.
    Then
    \[
        \lVert v(t) \rVert_A^2 = -\frac{d}{dt} E(\gamma(t)).
    \]
\end{lem}

\begin{proof}
    Using the right-invariance of the Riemannian metric on $\cG$ we obtain
    \[
        \frac{d}{dt} E(\gamma) = \langle \nabla E(\gamma),\dot\gamma\rangle_{\gamma}
        = -\langle \dot\gamma,\dot\gamma\rangle_{\gamma}
        = -\langle Av, v\rangle = -\lVert v\rVert_A^2.
    \]
\end{proof}

\begin{teo}\label{teo:global-existence-finite}
    Let $E\colon\cG\to \mathbb{R}_+$ be such that the gradient vector field $\nabla E$ on $\cG$ is locally Lipschitz continuous.
    Then the gradient flow \eqref{eq:gradient-flow-finite-dim} admits a unique global solution $\gamma\colon [0,\infty) \to \cG$.
\end{teo}

\begin{proof}
    Since $\nabla E$ is locally Lipschitz, it follows from standard ODE theory (cf.\ \cite{lang1999}) that the flow \eqref{eq:gradient-flow-finite-dim} admits a local solution $\gamma\colon [0,\epsilon)\to \cG$.
    Let $t\in [0,\epsilon)$.
    Then, from the definition of the Riemannian distance and the Hölder inequality we obtain
    \[
        d(\gamma(0),\gamma(t)) \leq \int_0^t \lVert v(s) \rVert_A ds
        \leq \sqrt{t} \left(\int_0^t \lVert v(s)\rVert^2_A ds \right)^{1/2}.
    \]
    Using Lemma~\ref{lem:estimate} this yields
    \[
        d(\gamma(0),\gamma(t)) \leq \sqrt{t \big( E(\gamma(0)) - E(\gamma(t)) \big)} \leq \sqrt{t E(\gamma(0))}.
    \]
    This means that $\gamma$ is $\frac{1}{2}$-Hölder continuous on $[0,\epsilon)$.
    In particular, we cannot have blowup as $t\to\epsilon$.
    It then follows by the general theory of maximal solutions of ODEs that the solution can be extended to $[0,\infty)$.
\end{proof}

\section{Sobolev setting for shape analysis}\label{section:3}
The previous section discusses gradient flows on Lie groups in an abstract setting.
These results are certainly valid for finite-dimensional Lie groups, but more generally also for the category of Hilbert Lie groups with strong Riemannian metrics (cf.~\cite{lang1999}).
However, the group $\cD(M)$ of smooth diffeomorphisms is not a Hilbert Lie group; it is a Fréchet Lie group (cf.~\cite{hamilton1982inverse}).

The results from \autoref{sec:gradflows} are not valid in the category of Fréchet manifolds, essentially because the fixed point theorem is not valid there.
The trick is to instead work in a Hilbert manifold setting via Sobolev completions of $\cD(M)$. 
The resulting completion $\cD^s(M)$ is a Hilbert manifold if $s> \operatorname{dim}(M)/2+1$.
But now the Lie group structure is lost; $\cD^s(M)$ is not a Lie group in the Hilbert manifold category (although it is a topological group).
For this reason Theorem~\ref{teo:local-existence} is not a trivial application of the results in \autoref{sec:gradflows}.

In the remaining part of this paper we give results to assert that the gradient vector field $\nabla E(\varphi)$ extended to $\cD^s(M)$ is 
\begin{enumerate}
    \item given as in Proposition~\ref{teo:gradient-functional}; 
    \item a smooth vector field on $\cD^s(M)$.
\end{enumerate}
We can then prove Theorem~\ref{teo:local-existence} via the extension of the Picard-Lindelöf theorem to Hilbert manifolds.
Our approach is a gradient flow analog of the approach to geodesic flows on $\cD^s(M)$ developed by Ebin and Marsden~\cite{ebin1970groups}.

We start in this section with the background material needed for a rigorous treatment of the infinite-dimensional Sobolev setting.
More details can be found in \cite{ebin1970groups,bruveris2017completeness}.

\subsection{Diffeomorphisms of Sobolev class}\label{S-diffeo-section}
Let $M$ be an oriented, compact, $n$-dimensional and $C^\infty$-smooth Riemannian manifold (without boundary) and fix $s$ such that $	s > n/2 +1.$ 
The subset $\cD^s(M)$ of the Sobolev space $H^s(M, M)=W^{s,2}(M,M)$ formed by $H^s$ auto\-morphisms which are $C^1$-diffeomorphisms is a Hilbert manifold and an open set in $H^s(M,M)$, see \cite[\textsection2]{ebin1970groups}.
Moreover, it is a topological group under the composition map and inversion. 
The tangent space at the identity, denoted by $T_{\id}\cD^s(M)$, is the space
of all $H^s$-smooth vector fields on $M$. Such vector fields are $H^s$-smooth sections of the tangent bundle $\X^s(M)$  \cite{ebin1970groups}. The tangent bundle of $\cD^s(M)$ is constructed as the subset of $H^s(M, TM)$ with elements that give elements of $\cD^s(M)$ when composed with the natural projection $\pi:TM\to M$. Precisely,
\begin{equation}
\label{tangentspace}
T_\varphi\cD^s(M) = \left\{U\in H^s(M,TM): \pi\circ U=\varphi\right\}.
\end{equation}
Writing $U=\xi\circ\varphi$ for some $\xi\in\X^s(M)$ we may identify vectors on the tangent spaces of $\cD^s(M)$ by the set $\left\{\xi\circ \varphi :\xi\in T_{\id}\cD^s(M)\right\}.$
It turns out that such compositions with diffeomorphisms provide sufficient control of the smoothness.
Let $N$ be a compact $n$-manifold and consider the maps $\alpha_\varphi: H^s(M, N) \to H^s(M, N)$ and $\omega_\psi: \cD^s(M)\to H^s(M, N)$ defined by 
\begin{equation}
 \alpha_\varphi(\psi) = \psi\circ \varphi\quad \text{and}\quad \omega_\psi (\varphi) = \psi\circ \varphi,
\end{equation}
where $\varphi\in\cD^s(M)$ and $\psi\in H^s(M, N)$.

\begin{lem}[\cite{ebin1970groups} ``Alpha Lemma.''] 
\label{alpha-lemma}
The mapping $\alpha_\varphi$ is $C^\infty$ smooth. Its tangent map is given by $(\psi,\dot\psi)\mapsto (\psi\circ\varphi,\dot\psi\circ\varphi)$.
\end{lem}

\begin{lem}[\cite{ebin1970groups} ``Omega Lemma.''] 
\label{omega-lemma}
The mapping $\omega_\psi$ is $C^l$ if $\psi\in H^{s+l}(M,N)$.
\end{lem}

In particular, the left and right translation mappings on the group of diffeomorphisms are continuous and smooth, respectively.
The fact that left translation is only continuous is the reason why $\cD^s(M)$ is a topological group but not a Lie group.
Right translation is smooth and provides a right-invariant inner product at any point $\varphi\in\cD^s(M)$ by the formula
\begin{equation}\label{eq:D-metric}
	(\xi,\eta)_e=(\xi\circ\varphi,\eta\circ\varphi)_\varphi\qquad \text{where } \xi,\eta\in T_e\cD^s(M)=\X^s(M),
\end{equation}
and where $(\,,)_e$ is an inner product on $\X^s(M)$.
We may now work in the so-called ``smooth dual.'' 
This dual is defined as follows: given an inertia operator $A\colon\X^s(M)\to\X^{s-k}(M)$ of $k$th order we identify $\X^{s-k}(M)$ as the dual of $\X^s(M)$.
Then, if $A$ is positive, it induces an inner product of the form~(\ref{eq:inertia_operator}). 

The inverse map $\varphi\mapsto\varphi^{-1}$ for $\cD^s(M)$ is continuous, or $C^l$ as a map $\cD^{s+l}(M)\to\cD^s(M)$. If $t\mapsto\psi(t)$ is a $C^1$ curve in $\cD^{s+l}(M),\ l\geq 1$, then
\begin{equation}
\label{eq:d-id}
    0=\frac{d}{dt}\left(\psi(t)^{-1} \circ\psi(t)\right)=\left(\frac{d}{dt}\psi(t)^{-1} \right)\circ\psi(t)+D\psi^{-1}(t)\circ\left(\frac{d}{dt}\psi(t)\right).
\end{equation}
These computations will we used in \textsection\ref{sec:mom-map-on-function} to derive our main results.

As right translation is a smooth map we may additionally define right-invariant vector fields on $\cD^s(M)$. Indeed, given $\xi\in\X^{s+l}(M)$, $l\geq1$, the composition $\xi\circ\varphi$ with $\varphi\in\cD^s(M)$ is right invariant.
The vector fields $\varphi\mapsto \xi\circ\varphi$ respect the vector field commutator \cite[p. 91]{marsden1970}, which motivates the nomenclature ``Lie-algebra'' for the vector space $T_{\id}\cD^s(M)$.


\subsection{The space of Riemannian metrics}\label{space-Rmetrics-section}
We denote by \edit{$S_{0,2}(M)$} the bundle of symmetric \edit{$(0,2)$-tensor fields} on $M$.
\edit{Likewise, $S_{2,0}(M)$ denotes symmetric $(2,0)$-tensor fields.}
Let $s>n/2+1$ as before and \edit{denote by $S_{0,2}^{s-1}(M)$ and $S_{2,0}^{s-1}(M)$ the Hilbert spaces obtained by $H^s$ completion.}
We write $C^0\Met(M)$ for the space of continuous Riemannian metrics and define the set of $H^{s-1}$-smooth Riemannian metrics as $\Met^{s-1}(M)=S_{0,2}^{s-1}(M)\cap C^0\Met(M)$. Since $H^{s-1}\subseteq C^0(S_{0,2}(M))$ and the embedding \edit{is} continuous, we have that $\Met^{s-1}(M)$ is open in $S_{0,2}^{s-1}(M)$. 
In fact, it is an open, convex, positive cone in $S_{0,2}^{s-1}(M)$. 
The space $\Met^{s-1}(M)$ is a smooth Hilbert manifold \cite{smolentsev2007spaces}, and $T\Met^{s-1}(M) \simeq \Met^{s-1}(M)\times S_{0,2}^{s-1}(M)$. 

There is a canonical metric on the space of metrics, for which one can compute the geodesics explicitly \cite{FrGr1989,gil1992the}, but we will use another metric for our purposes. 
Indeed, the following defines an $L^2$-type distance on $\Met^{s-1}(M)$:
\begin{equation}
    \|h_1-h_2\|_{L^2}^2 = \int_{M} \left[(h_1)^{ij}-(h_2)^{ij}\right] \left[(h_1)_{ij}-(h_2)_{ij}\right]\, d\mu
\end{equation}
for $h_1,h_2\in \Met^{s-1}(M)$.

%
%

%

\section{Gradient flow on diffeomorphisms}\label{section:gradflow-diffeos}
In this section we turn our attention to gradient flows of the form in \autoref{sec:gradflows} for the special case where 
\begin{equation}\label{special-case}
    \begin{split}
    & \cG=\cD^s(M) \\
    & Q=H^s(M,\R)\times \Met^{s-1}(M).
    \end{split}
\end{equation}
%
Indeed, we consider the gradient flow $\dot{\vp}= -\nabla E(\vp)$, with $\vp(0)= \id$, where $\id$ is the identity mapping on $M$ and where 
\begin{align*}
    E(\varphi) =\underbrace{\frac{1}{2}\|I_0\circ\varphi^{-1}-I_1\|_{L^2}^2}_{E_1(\varphi) = f_1(I_0\circ\vp^{-1})} +  \underbrace{\frac{\sigma}{2}\|\vp_*g-g\|_{L^2}^2}_{E_2(\varphi) = f_2(\vp_*g)}
\end{align*}
for $\sigma>0$ constant.
Recall that the template and the source are fixed functions $I_0, I_1 \in H^s(M,\mathbb R)$.
We show that the gradient flow equation for this energy functional is as stated in equation \eqref{eq:gradient-flow-explicit}, and that it fulfills the geometric structure stated in Proposition~\ref{teo:gradient-functional}, as an infinite-dimensional example of the abstract gradient flow structure on Lie groups discussed in \autoref{sec:gradflows}.
Notice that the energy functional splits into two components, $f_1$ and $f_2$, corresponding to the action of $\cD^s(M)$ on the two components of $Q$.

For concreteness, let us restate all the components of the flow explicitly:

\edit{
\begin{prop}[\emph{cf.}~Proposition~\ref{teo:gradient-functional}] \label{prop:gradient_flow}
The gradient system on $G=\cD^s(M)$ is given as

\begin{subequations}\label{gradient_explicit}
    \begin{align}
        \dot{\vp}&= -\nabla E(\varphi) = -v\circ\vp, \quad Av = J\Big(
            I_0\circ\vp^{-1}, \vp_*g, \underbrace{I_0\circ\vp^{-1}-I_1}_{df_1(I_0\circ\vp^{-1})}, \underbrace{\vp_*g-g}_{df_2(\vp_*g)}
        \Big) 
    \end{align}
    where $A=(1-\alpha\Delta)^k$ as in \eqref{eq:inertia_operator} and the momentum map $J$ is given by
    \begin{equation}
        J(I, h, P, p) = \underbrace{-P\nabla I}_{J_1(I,P)} + \underbrace{2\operatorname{div}_h(p)}_{J_2(h,p)}
    \end{equation}
\end{subequations}%
with $\operatorname{div}_h(\alpha) =  \mathrm{tr}_h(\operatorname{div}\alpha) + \mathrm{tr}_{V,W}\left[ \alpha(V,W)\nabla h(V,\cdot,W)-\alpha(V,W)\nabla h(\cdot,V,W)/2\right]$ for  $h=\varphi_*g\in \Met^{s-1}(M)$  and $\nabla h(U,V,W):=\nabla_U h(V,W)$.

\end{prop}
}

%

To verify the equations \eqref{gradient_explicit} we now extend the computations in the proof of Proposition~\ref{teo:gradient-intro} to the Sobolev space setting.
Thus, we need to compute the momentum map $J$.
Since $Q$ is a direct product, it follows that the momentum map $J$ is decomposed into the two momentum maps: $J_1$, for the action on functions, and $J_2$, for the action on Riemannian metrics.
 

\subsection{Momentum map for the action on functions}\label{sec:mom-map-on-function}
Recall that the push-forward action of $\cD^s(M)$ on $I\in H^s(M,\R)$ is $\varphi.I \coloneqq I\circ\varphi^{-1}$. 
The associated infinitesimal action of $\xi\in \mathfrak{X}^s(M)$ is $\xi.I=-\iota_\xi d I$.
We work in the \emph{smooth dual} of $H^s(M,\R)$ by considering only functionals on $H^s(M,\R)$ whose variational derivative is an element of $H^s(M,\mathbb{R})$.
In particular, this is true for the functional \edit{$f_1(I) = \frac{1}{2}\lVert I-I_1\rVert^2$} in the equations~\eqref{gradient_explicit}.
\edit{
On the other hand, since the infinitesimal action involve derivatives, it is \emph{not} enough to restrict $\mathfrak{X}^s(M)^*\simeq \mathfrak{X}^{-s}(M)$ to $\mathfrak{X}^s(M)$.
Instead, the suitable restriction of $\mathfrak{X}^s(M)^*$ is to $\mathfrak{X}^{s-1}(M)$.
}

The momentum map for the infinitesimal action on $H^s(M,\R)$ is defined by
\begin{equation*}
    \langle J_1(I,P), \xi \rangle = \langle P, \xi.I\rangle ,\qquad \forall\, v\in\mathfrak{X}^s(M).
\end{equation*}
A direct calculation and the Sobolev embedding theorem yield the following result.

\begin{lem} \label{prop:Hs-momentum-map}
    Let $s>1+n/2$.
    Then the momentum map for the cotangent lift of the action 
    \begin{equation*}
        \cD^{s}(M)\times H^s(M,\R)\ni (\varphi,I)\mapsto I\circ\vp^{-1}\in H^s(M,\R)
    \end{equation*}
    is a smooth mapping $J_1\colon H^s(M,\R)\times H^s(M,\R)\to \X^{s-1}(M)$ given by
    \begin{equation}
        J_1(I,P)= -P\nabla I.
    \end{equation}

\end{lem}

\subsection{Momentum map for the action on metrics}\label{sec:on-metrics}
The push-forward defines an action of $\cD^s(M)$ on the space of metrics $\Met^{s-1}(M)$ by

\begin{equation*}
    \begin{aligned}
    	&\cD^s(M)\times\Met^{s-1}(M)\ni (\varphi,g)\mapsto \varphi_*g\in \Met^{s-1}(M) \\
    	&(\varphi_*g)_x(u,v)=g_{\varphi^{-1}(x)}\left(d\varphi^{-1}u,d\varphi^{-1}v\right),\quad x\in M,\ u,v\in T_{x}M
    \end{aligned}
\end{equation*}%
The corresponding infinitesimal action of $\xi\in\X^s(M)$ is given by minus the Lie derivative: $\xi . g = -\cL_\xi g$.

Analogous to the previous section, we now compute 
the properties of the momentum map that we require for the proof of Theorem \ref{teo:local-existence}. 
In particular, we compute $J_2$ in coordinates on $M$. 
The raising and lowering of indices (musical isomorphisms) are always with respect to the fixed Riemannian metric $g$.

\edit{In this section, we restrict $\X^s(M)^*$ to $\X^{s-2}(M)$.}
Notice that this is different from the setting in \autoref{sec:mom-map-on-function}: one extra derivative is needed.
The reason is that the action of $\cD^s(M)$ on $\Met^{s-1}(M)$ requires a derivative on the inverse of the diffeomorphism, but the action on functions does not.



\begin{lem}
\label{prop:Met-momentum-map}
Let $s>2+n/2$.
Then the momentum map for the cotangent lift of the action
$$ \cD^{s}(M)\times \Met^{s-1}(M)\ni (\varphi,g)\mapsto \varphi_*g\in\Met^{s-1}(M)$$
is a smooth mapping $\Met^{s-1}(M)\times S_{2,0}^{s-1}(M)\to \X^{s-2}(M)$ given by
\begin{equation}
 J_2(h,p)=2\mathrm{tr}_h(\operatorname{div}p)+ \mathrm{tr}_{V,W}\left[2 p(V,W)\nabla h(V,\cdot,W)-p(V,W)\nabla h(\cdot,V,W)\right],
\end{equation}
where $\nabla h(U,V,W):=\nabla_U h(V,W)$.
\end{lem}

\begin{proof}
The infinitesimal action on $h\in\Met^{s-1}(M)$ is $\xi.h= -\cL_\xi(h)$.
In local coordinates we have
\begin{align}
    \langle J_2(h,p),V\rangle 
    &\coloneqq  \langle p,-\cL_Vh\rangle\\
    &=  -\int_M p^{ij}(\cL_Vh)_{ij} d\mu.
\end{align}
In these local coordinates, the Lie derivate can be expressed as
\begin{equation}
    (\cL_V h)_{ij} = V^m\nabla_m h_{ij}+ (\nabla_i V^m)h_{mj}+(\nabla_j V^m) h_{im},
\end{equation}
where $\nabla$ is the Levi-Civit\`a connection associated with $g$.
Therefore, we compute using integration by parts
\begin{align}
    J_2(h,p)(V) = & -\int_M p^{im}(\cL_Vh)_{im} d\mu\\
    = & -\int_M p^{im}V^j\nabla_j h_{im} d\mu  -2\int_M p^{im}(\nabla_i V^j)h_{jm} d\mu\\
    = & -\int_M p^{im}V^j\nabla_j h_{im} d\mu  +2\int_M (\nabla_ip^{im} h_{jm})V^j d\mu\\
    = & -\int_M p^{im}V^j(\nabla_j h_{im})d\mu +2\int_M p^{im}(\nabla_i h_{jm})V^j d\mu \\ 
    & \qquad +2\int_M  (\nabla_i p^{im})h_{jm}V^j d\mu \label{last-term}
\end{align}
That is,
\begin{equation*}
\label{J-in-coords}
    J_2(h,p)_j = 2  (\nabla_i p^{im})h_{jm}  +2p^{im}(\nabla_i h_{jm})-p^{im}(\nabla_j h_{im})
\end{equation*}
This provides the formula in the proposition. This is a smooth mapping, since $s>n/2+2$ and since the covariant derivative is a smooth mapping of tensors.
\end{proof}


%
%
%

\section{Smoothness of the gradient flow} \label{exist-and-smoothness-section}

We have thus far discussed the geometry of the gradient flow. 
In this section we turn to the question of existence of the flow, leading up to a proof of Theorem~\ref{teo:local-existence}.

Define
\begin{equation}
    \begin{aligned}
        \tau\colon  \cD^{s}(M) &\to H^s(M,\R)\times\Met^{s-1}(M)\times H^s(M,\R)\times S_2^{s-1} \\
         \vp &\mapsto \big(I_0\circ\vp^{-1}, \vp_*g, I_0\circ\vp^{-1}-I_1, \sigma(\vp_*g-g)\big).
    \end{aligned}
\end{equation}%
Combined with the momentum map $J$ obtained in Lemma~\ref{prop:Hs-momentum-map} and Lemma~\ref{prop:Met-momentum-map} this yields the mapping
\begin{equation}
\label{eq:F-mapping}
\begin{split}
    F \colon \cD^{s}(M) &\to \X^{s-2}(M)\\
    \vp &\mapsto F(\vp)\coloneqq (J\circ \tau)(\vp).
\end{split}
\end{equation}
The gradient flow \eqref{gradient_explicit} can be written in terms of this mapping as 
\begin{equation}\label{eq:Frhs}
    \dot\vp = -(A^{-1}F(\vp))\circ\vp .
\end{equation}
The coordinate expression~\eqref{J-in-coords} for $J_2$ gives
\begin{equation}
\begin{aligned} 
      F(\varphi)_k =& - (I_0\circ\vp^{-1}-I_1)\nabla_k(I_0\circ\vp^{-1}) + 2\sigma g^{il}g^{jm}(\varphi_*g-g)_{lm}\nabla_i(\varphi_*g)_{kj} + \\ & 2\sigma g^{il}g^{jm}(\varphi_*g-g)_{lm}\nabla_i(\varphi_*g)_{kj} -\sigma g^{il}g^{jm}(\varphi_*g-g)_{lm}\nabla_k(\varphi_*g)_{ij}.
\end{aligned}
\end{equation}
%
The pushforward $\vp_*g$ uses one derivative of $\vp^{-1}$ so by counting derivatives we get that $F$ is a second-order non-linear differential operator in $\vp^{-1}$; this observation is the first key in the existence analysis.

The second key is to factorize the right-hand side of \eqref{eq:Frhs} as a composition of two smooth mappings.
To this end, let $T\cD^{s-k}(M)\upharpoonright \cD^{s}(M)$ denote the restriction of the tangent bundle $T\cD^{s-k}(M)$ to the base $\cD^{s}(M)$, and define $\tilde{F}$ by 
\begin{equation}
\label{eq:J-tilde}
\begin{split} 
    \tilde{F} :\cD^{s}(M) &\to T\cD^{s-2}(M)\upharpoonright\cD^{s}(M)\\
    \varphi &\mapsto \left(\varphi, F (\varphi)\circ\varphi\right).
\end{split}
\end{equation}

\begin{lem}
\label{lem:F-tilde}
Let $s>2+n/2$.
Then the mapping $\tilde{F}$ is smooth.
\end{lem}

\begin{proof}
Expressed in local coordinates
\begin{multline*}
    F(\varphi)(x)_k = f_k\left(\varphi^{-1}(x)_1,\ldots,\varphi^{-1}(x)_n, \frac{\partial\varphi^{-1}_1}{\partial x_1}(x),\ldots,\frac{\partial\varphi^{-1}_n}{\partial x_n}(x),\right. \\ \left. \frac{\partial^2\varphi^{-1}_1}{\partial x_1\partial x_1}(x),\ldots,\frac{\partial^2\varphi^{-1}_n}{\partial x_n\partial x_n}(x)\right)
\end{multline*}
for smooth functions $f_k$ defined on some open subset of $\mathbb{R}^n\times \mathbb{R}^{n\times n}\times \mathbb{R}^{n\times n\times n}$ with $n=\dim M$.
The mapping 
\[
    \cD^s(M) \ni \varphi \mapsto \frac{\partial \varphi^{-1}_i}{\partial x_j}\circ\varphi \in H^{s-1}(M)
\]
is smooth, because $D\varphi^{-1}\circ\varphi (x) = (D\varphi(x))^{-1}$ and matrix inverse is a smooth, point-wise operation.
From the chain rule notice that
\[
    D^2\varphi^{-1}\circ\varphi(x) = D(D\varphi^{-1}\circ\varphi)(x)(D\varphi(x))^{-1}= D(D\varphi)^{-1}(x)(D\varphi(x))^{-1}.
\]
The mapping
\[
    \mathcal{D}^s \ni \varphi \mapsto \frac{\partial^2 \varphi^{-1}_i}{\partial x_{j}\partial x_k}\circ\varphi \in H^{s-2}(M)
\]
is smooth if $s>2+n/2$ (so that $D(D\varphi)^{-1}$ is above the Sobolev embedding threshold and therefore can be smoothly multiplied with elements in $H^{s-1}(M)$).
Since $f_k$ are smooth mappings, the full composition $\tilde F(\varphi) = (\varphi, F(\varphi)\circ\varphi)$ is smooth by the Omega Lemma \ref{omega-lemma}.
\end{proof}

\begin{rem}
    By a small modification we see that whenever $F$ is a smooth differential operator of order $k$ in $\vp^{-1}$, then Lemma~\ref{lem:F-tilde} is valid for $s>k+n/2$.
\end{rem}

We stress that the result in Lemma~\ref{lem:F-tilde} is quite remarkable, given that the inversion in $\cD^s(M)$ is not smooth (only continuous), so the mapping $F$ in \eqref{eq:J-tilde} is certainly \emph{not} smooth; it is the post-composition with $\vp$ together with the chain-rule that saves the situation for $\tilde F$.
However, we have only gone halfway: we now need to formulate the gradient flow in terms of $\tilde F$ to exploit the smoothness.
To accomplish that we repeat for the inertia operator $A$ what we did for $F$.
Indeed, define $\tilde A$ by
\begin{equation}
    \begin{aligned}
        \tilde A\colon T\cD^s(M) &\to T\cD^{s-2}(M)\upharpoonright\cD^{s}(M)\\
        (\varphi,\dot\varphi) &\mapsto \left(\varphi, (A(\dot\vp\circ\vp^{-1}))\circ\vp\right).
    \end{aligned}
\end{equation}
Then the gradient flow \eqref{gradient_explicit} can be written
\begin{equation}\label{eq:gradient_flow_smooth_formulation}
    \dot\vp = - (\tilde A^{-1}\circ \tilde F)(\varphi).
\end{equation}
Composition of smooth maps is smooth, and we already know that $\tilde F$ is smooth, so it remains to prove that $\tilde A$ is smooth.
For this, we use a small modification of a result by Ebin and Marsden~\cite{ebin1970groups} (see also \cite[Lemma~3.2]{modin2015generalized}).


\begin{lem} 
\label{lem:bundle-iso}
	Let $B:\X^s(M)\to\X^{s-q}(M)$ be an elliptic differential operator of order $q$ and  let $s>q+n/2$. 
    Assume further that $B$ is an isomorphism.
    Then the operator $\tilde{B}:T\cD^s(M)\to T\cD^{s-q}(M)\upharpoonright \cD^{s}(M)$ given by
    \begin{equation*}
        \tilde B \colon (\vp,\dot\vp)\mapsto\Big(\varphi,\big(B(\dot{\varphi}\circ\vp^{-1})\big)\circ\vp\Big)
    \end{equation*} 
    is a smooth vector bundle isomorphism.	
\end{lem}

\begin{proof}
	The operator $B$ is a smooth isomorphism from vector fields of class $H^s$ to vector fields of class $H^{s-q}$ and $\tilde{B}$ is a smooth bundle map, see \cite{ebin1970groups}. 
	It remains to show that there exists an inverse.
	Let us use the notation
	$$
	\tilde{B}(\varphi,\dot\varphi)=(\tilde{b_1}(\varphi,\dot{\varphi}),\tilde{b_2}(\varphi,\dot{\varphi})).
	$$
	Then, the (total) derivative of $\tilde{B}$ is the matrix
	$$
	\begin{bmatrix}
	  \partial\tilde{b_1}/\partial\varphi &  \partial\tilde{b_1}/\partial\dot\varphi \\
	  \partial\tilde{b_2}/\partial\varphi & \partial\tilde{b_2}/\partial\dot\varphi 
	\end{bmatrix}
	$$
	For instance, writing $\dot\varphi=\xi\circ\varphi$ we have $\tilde{b}_2(\varphi,\dot\varphi)= \left(B(\dot{\varphi}\circ\varphi^{-1})\right)\circ\varphi = TR_\varphi (B\xi).$ 
	Thus,
	$$
	\frac{\partial \tilde{b_2}}{\partial \varphi} =D (B\xi)\circ\varphi =TR_\varphi D(B\xi).
	$$
	At $\varphi\in\cD^s(M)$, the derivative of $\tilde{B}$ is  
	\begin{equation}
	\label{eq:dA}
	\begin{bmatrix}
	  {\id}_{\cD^s(M)} &  0 \\
	  D(B\xi)\circ\varphi & B_\varphi 
	\end{bmatrix},
	\end{equation}
	This is lower triangular and therefore invertible. It is regular.
	By the inverse mapping theorem for Banach spaces, see~\cite{lang1999}, $\tilde{B}$ is a smooth isomorphism. Indeed, the inverse of~(\ref{eq:dA}) is a (smooth) map given by
	\begin{equation}
	\begin{bmatrix}
	{\id}_{\cD^s(M)} &  0 \\
	-(B_\varphi)^{-1}\circ D(B\xi)\circ\varphi & (B_\varphi)^{-1}
	\end{bmatrix}
	\end{equation}
	where $(B_\varphi)^{-1}=TR_\varphi\circ B^{-1}\circ TR_{\varphi^{-1}}$.
\end{proof}


We are now ready to prove Theorem \ref{teo:local-existence}.

\begin{proof}[Proof of Theorem \ref{teo:local-existence}]
    Recall the key: that we can factorize $\nabla E(\vp)$ as
    \begin{equation*}
        \nabla E(\vp) = (\tilde A^{-1}\circ \tilde F)(\varphi).
    \end{equation*}
    From the two previous lemmata the mapping $\nabla E$ is a smooth vector field on $\cD^s(M)$ (since the inertia operator $A=(1-\alpha \Delta)^k$ for $\alpha>0$ and $k\in\mathbb{N}$ fulfills the requirements of Lemma~\ref{lem:bundle-iso}).
    Given that $\cD^s(M)$ is a Hilbert manifold, the existence of a local flow for equation~(\ref{eq:grad-flow-intro}) up to a maximal time $T>0$ follows from the Picard-Lindelöf theorem of ODE theory on Banach spaces (cf.\ \cite[Ch.~IV]{lang1999}).
    This theory also implies that the local solution depends smoothly on the initial conditions.

    For the global result, we follow precisely the same steps as in the finite-dimensional case in Theorem~\ref{teo:global-existence-finite}.
    If $k\geq s$ then the Riemannian metric defined by the inertia operator $A$ is strong enough to dominate the topology of $\cD^s(M)$.
    The estimates obtained as in Theorem~\ref{teo:global-existence-finite} thus exclude the possibility of blowup as $t\to T$.
    Since the general theory of maximal solutions of ODEs is valid also in the Banach category, we conclude that $T=\infty$ in this case.
\end{proof}

\bibliographystyle{amsalpha-cjk} 
\bibliography{refs.bib}
\vspace{2em}

\end{document}